\documentclass[11pt]{amsart}

\usepackage{amsmath}
\usepackage{amssymb}
\usepackage{bbm}
\usepackage{esint} %integrals
\usepackage[colorlinks,citecolor=red,pagebackref,hypertexnames=false]{hyperref}
\usepackage{esint} %integrals
\usepackage{enumitem} %list with alphabet
\usepackage{verbatim} %for \begin{comment}\end{comment}
\usepackage{ifpdf}
\ifpdf
\usepackage{pgf}
\newcommand{\R}{{\mathbb R}}       % Field of real numbers

\newcommand{\HH}{{\mathcal H}}

\newcommand{\cV}{{\mathcal V}}

\newcommand{\UU}{{\mathcal U}}

\newcommand{\diam}{{\rm diam}}
\newcommand{\dist}{{\rm dist}}

\newcommand{\rf}[1]{{(\ref{#1})}}

\newcommand{\supp}{\operatorname{supp}}

\newcommand{\vphi}{{\varphi}}

\newcommand{\vv}{{\vspace{2mm}}}
\newcommand{\vvv}{{\vspace{3mm}}}
\newcommand{\wt}[1]{{\widetilde{#1}}}

\newcommand{\vmo}{{\operatorname{VMO}}}

\def\Xint#1{\mathchoice
{\XXint\displaystyle\textstyle{#1}}%
{\XXint\textstyle\scriptstyle{#1}}%
{\XXint\scriptstyle\scriptscriptstyle{#1}}%
{\XXint\scriptscriptstyle\scriptscriptstyle{#1}}%
\!\int}
\def\XXint#1#2#3{{\setbox0=\hbox{$#1{#2#3}{\int}$ }
\vcenter{\hbox{$#2#3$ }}\kern-.58\wd0}}

\def\avint{\;\Xint-}

\usepackage{pgf,tikz} %tikz: dibuixos i esquemes, com el cercle

\definecolor{ffffff}{rgb}{1.0,1.0,1.0}
\definecolor{qqqqff}{rgb}{0.0,0.0,1.0}
\definecolor{ffqqqq}{rgb}{1.0,0.0,0.0}
\definecolor{zzzzqq}{rgb}{0.6,0.6,0.0}
\definecolor{marronet}{rgb}{0.6,0.2,0}
\definecolor{negre}{rgb}{0,0,0}
\definecolor{vermell}{rgb}{0.8,0.05,0.05}
\definecolor{blau}{rgb}{0.3,0.2,1.}
\definecolor{blauclar}{rgb}{0.,0.,1.}
\definecolor{grisfosc}{rgb}{0.25098039215686274,0.25098039215686274,0.25098039215686274}
\definecolor{verd}{rgb}{0.1,0.6,0.1}
\definecolor{taronja}{rgb}{0.9,0.6,0.05}
\definecolor{vermellclar}{rgb}{1.,0.,0.}
\definecolor{verdet}{rgb}{0,0.8,0.1}
\definecolor{blauverd}{rgb}{0,0.4,0.2}
\definecolor{grisclar}{rgb}{0.6274509803921569,0.6274509803921569,0.6274509803921569}

\newtheorem{theorem}{Theorem}[section]
\newtheorem{lemma}[theorem]{Lemma}

\newtheorem*{theorem*}{Theorem}
\newtheorem*{theorema}{Theorem A}
\newtheorem*{theoremb}{Theorem B}

\theoremstyle{definition}

\theoremstyle{remark}
\newtheorem{rem}[theorem]{\bf Remark}

\numberwithin{equation}{section}

\newcommand{\brem}{\begin{rem}}
\newcommand{\erem}{\end{rem}}

\usepackage{xcolor}

\begin{document}

\title[Counterexample for two-phase problem for harmonic measure]{A counterexample regarding a two-phase problem for harmonic measure in VMO}

\author{Xavier Tolsa}

\address{ICREA, Barcelona\\
Dept. de Matem\`atiques, Universitat Aut\`onoma de Barcelona \\
and Centre de Recerca Matem\`atica, Barcelona, Catalonia.}
\email{xavier.tolsa@uab.cat}



\thanks{The author was supported by the European Research Council (ERC) under the European Union's Horizon 2020 research and innovation programme (grant agreement 101018680). 
 Also partially supported by MICINN (Spain) under the grant PID2020-114167GB-I00, the María de Maeztu Program for units of excellence (Spain) (CEX2020-001084-M), and 2021-SGR-00071 (Catalonia). }

\begin{abstract}
Let $\Omega^+\subset\R^{n+1}$ be  a vanishing Reifenberg flat domain such that $\Omega^+$ and $\Omega^-=\R^{n+1}\setminus\overline {\Omega^+}$
 have joint big pieces of chord-arc subdomains and such that the outer unit normal to $\partial\Omega^+$ belongs to $\vmo(\omega^+)$, 
  where $\omega^\pm$ is the harmonic measure in $\Omega^\pm$. Up to now it was an open question if these conditions imply that $\log\dfrac{d\omega^-}{d\omega^+} \in \vmo(\omega^+)$. 
In this paper we answer this question in the negative by constructing an appropriate counterexample in $\R^2$, with the
additional property that
 the outer unit normal to $\partial\Omega^+$ is constant $\omega^+$-a.e.\ in $\partial\Omega^+$.
\end{abstract}

\maketitle

\section{Introduction}

This paper deals with a two-phase problem for harmonic measure. In such problems one considers two disjoint domains $\Omega^+,\Omega^-\subset\R^{n+1}$ whose boundaries have non-empty intersection, and whose respective harmonic measures $\omega^+,\omega^-$ are usually mutually absolutely continuous in some subset of $\partial\Omega^+\cap \partial\Omega^-$. 
Then one has to relate
 the analytic properties of the
density $\frac{d\omega^-}{d\omega^+}|_{\partial\Omega^+\cap \partial\Omega^-}$ to the geometric properties of $\partial\Omega^+\cap \partial\Omega^-$. For example, in \cite{AMT-cpam} and \cite{AMTV} it has been proved that if
$\omega^+$ and $\omega^-$ are mutually absolutely continuous in a subset $E\subset\partial\Omega$, then $\omega^\pm|_E$
is concentrated in an $n$-rectifiable set. 
For a previous related result see \cite{KPT}, and for a more recent work involving elliptic measure,
 see \cite{AM-blow}.  
 For other works of more quantitative nature where one assumes $\Omega^+$, $\Omega^-$ to be complementary NTA domains and either that $\omega^-\in A_\infty(\omega^+)$ or stronger conditions, see \cite{Kenig-Toro-crelle}, \cite{Engelstein}, \cite{AMT-quantcpam}, 
\cite{Prats-Tolsa}, and \cite{Tolsa-Toro}, for instance. See also \cite{BET1} and \cite{BET2} for other recent results
dealing with the structure of the singular set of the boundary.

In connection with the precise question studied in this paper, by combining works of Prats, Tolsa, and Toro, the following is known:

\begin{theorema} 
Let $\Omega^+\subset\R^{n+1}$ be a bounded NTA domain and let $\Omega^-= \R^{n+1}\setminus \overline{\Omega^+}$ be an NTA domain as well.
Denote by $\omega^+$ and $\omega^-$ the respective harmonic measures with poles $p^+\in\Omega^+$ and $p^-\in\Omega^-$. Suppose that $\Omega^+$ is a $\delta$-Reifenberg flat domain, with $\delta>0$ small enough.
Then the following conditions are equivalent:
\begin{itemize}
\item[(a)] $\omega^+$ and $\omega^-$ are mutually absolutely continuous and $\log\dfrac{d\omega^-}{d\omega^+} \in \vmo(\omega^+).$

\item[(b)] $\Omega^+$ is vanishing Reifenberg flat,  $\Omega^+$ and  $\Omega^-$ have joint big pieces of chord-arc subdomains, and 
\begin{equation}\label{eqnb*}
\lim_{\rho\to 0}\, \sup_{\substack{x\in\partial\Omega^+\\0<r\leq \rho}}\, \avint_{B(x,r)} |N - N_{B(x,r)}|\,d\omega^+ = 0, %\tag{*} 
\end{equation}
where $N$ is the outer unit normal to $\partial\Omega^+$ and
$N_{B(x,r)}$ is the unit normal (pointing to $\Omega^-$) to the $n$-plane $L_{B(x,r)}$ minimizing 
\begin{equation}\label{eqnb*2}
\max\Big\{\sup_{y\in \partial\Omega^+\cap B(x,r)}\dist(y,L_{B(x,r)}), \sup_{y\in L_{B(x,r)}\cap B(x,r)}\dist(y,\partial\Omega^+)\Big\}.
\end{equation}

\item[(c)] $\Omega^+$ and  $\Omega^-$ have joint big pieces of chord-arc subdomains and \rf{eqnb*} holds.\vv

\item[(d)] $\Omega^+$ is a chord-arc domain and the outer unit normal $N$ to $\partial\Omega^+$ belongs to $\vmo(\HH^n|_{\partial\Omega^+})$.

\end{itemize}
\end{theorema}

The equivalence (a)\;$\Leftrightarrow$\;(b) was proven in \cite{Prats-Tolsa}, while (c)\;$\Rightarrow$\;(d)\;$\Rightarrow$\;(a) was shown more recently in \cite{Tolsa-Toro}.
For the definitions of NTA, Reifenberg flat, and chord-arc domains, and VMO, see Section~\ref{secprelim}.
Remark that, for a Reifenberg flat domain (and more generally, for a two-sided NTA domain) the property of having joint big pieces of chord-arc subdomains is equivalent to the fact that $\omega^-\in A_\infty(\omega^+)$, as shown in \cite{AMT-quantcpam}.  This property implies that both $\omega^+$ and $\omega^-$ are concentrated in an $n$-rectifiable set, by \cite{AMT-cpam}, and so the outer unit normal $N$ is defined $\omega^\pm$-a.e.
A key point in the theorem is that a priori it is not assumed that $\partial\Omega^+$ has locally finite surface measure in (a), (b), and (c). Nevertheless, (c)\;$\Rightarrow$\;(d) ensures that $\HH^n(\partial\Omega^+)<\infty$ (recall that a chord-arc domain is an $n$-Ahlfors regular NTA domain).

It is worth comparing the previous theorem with the following result of Kenig and Toro about the one-phase problem for harmonic measure:

\begin{theoremb}[\cite{Kenig-Toro-duke}, \cite{Kenig-Toro-annals}, \cite{Kenig-Toro-AENS}]
Let $\Omega\subset\R^{n+1}$ be a bounded $\delta$-Reifenberg flat chord-arc domain, with $\delta>0$ small enough.
Denote by $\omega$ the harmonic measure in $\Omega$ with pole $p\in\Omega$ and write $\sigma=\HH^n|_{\partial\Omega}$. Then the following conditions are equivalent:
\begin{itemize}
\item[(i)] $\log\dfrac{d\omega}{d\sigma} \in \vmo(\sigma).$

\item[(ii)] $\Omega$ is vanishing Reifenberg flat and the outer unit normal $N$ to $\partial\Omega$ belongs to $\vmo(\sigma)$.

\item[(iii)] The outer unit normal $N$ to $\partial\Omega$ exists $\sigma$-a.e.\ and it belongs to $\vmo(\sigma)$.

\end{itemize}
\end{theoremb}

In view of the similarities between the statements (b) in Theorem A  and (ii) in Theorem B, a natural question (which is left open in the works \cite{Prats-Tolsa} and \cite{Tolsa-Toro}) is if under the assumptions in Theorem A, the statement (b) is equivalent to the following:
\begin{itemize}
\item[(b')] $\Omega^+$ is vanishing Reifenberg flat,  $\Omega^+$ and  $\Omega^-$ have joint big pieces of chord-arc subdomains, and the outer unit normal $N$ to $\partial\Omega^+$ belongs to $\vmo(\omega^+)$.
\end{itemize}
Notice that (b') is the same as (b), with $N_{B(x,r)}$ in \rf{eqnb*} replaced by $\avint_{B(x,r)} N\,d\omega^+$.
In this paper we provide a negative answer by constructing a suitable counterexample in $\R^2$.
The precise result is the following:

\begin{theorem}\label{teoex}
There exists a bounded vanishing Reifenberg flat domain $\Omega^+\subset\R^2$ such that $\Omega^+$ and $\Omega^-:=
\R^2\setminus\overline {\Omega^+}$ have joint big pieces of chord-arc subdomains, the outer unit normal $N$ is constant $\omega^+$-a.e.\ in $\partial\Omega^+$, and such that \rf{eqnb*} does not hold, and so $\log\dfrac{d\omega^-}{d\omega^+} \not\in \vmo(\omega^+).$
\end{theorem}

Observe that in the  theorem, although $\log\dfrac{d\omega^-}{d\omega^+} \not\in \vmo(\omega^+)$,  we still have $\omega^-\in A_\infty(\omega^+)$, because $\Omega^+$ and $\Omega^-$ have joint big pieces of chord-arc subdomains.
Of course, the theorem also implies that in the statement (c) in Theorem A one cannot replace the assumption
\rf{eqnb*} by the fact that $N\in \vmo(\omega^+)$.

To prove Theorem \ref{teoex} we will construct a snowflake-type domain satisfying the required properties. Its construction is vaguely inspired by the example in
\cite[Section 8]{AMT-quantcpam}. We will show that in this domain harmonic measure is concentrated in a countable collection of vertical segments by a probabilistic argument which uses the central limit theorem (see Lemma \ref{lemvertical}) and which has its own interest.

\vvv

% ********************************************************************************
% ********************************************************************************

\section{Preliminaries}\label{secprelim}

We denote by $C$ or $c$ some constants that may depend on the dimension and perhaps other fixed parameters. Their values may change at different occurrences. On the contrary, constants with subscripts, like $C_0$, retain their values.
For $a,b\geq 0$, we write $a\lesssim b$ if there is $C>0$ such that $a\leq Cb$. We write $a\approx b$ to mean $a\lesssim b\lesssim a$. If we want to emphasize that the implicit constant depends on some parameter $\eta\in\R$, we write $a\approx_\eta b$.

\subsection{Ahlfors regularity and uniform rectifiability}

All measures in this paper are assumed to be Borel measures.
A measure $\mu$ in $\R^d$ is called {\em doubling} if there is some constant $C>0$ such that
$$\mu(B(x,2r))\leq C\,\mu(B(x,r))\quad \mbox{ for all $x\in\supp\mu$ and $r>0$.}$$
A measure $\mu$ in $\R^{d}$ is called {\em Ahlfors regular} (or $n$-Ahlfors regular) if
\begin{equation}\label{eqAD1}
C^{-1}r^n\leq \mu(B(x,r))\leq C r^n \quad \mbox{ for all $x\in\supp\mu$ and $0<r\leq \diam(\supp\mu)$.}
\end{equation}
A set $E\subset\R^d$ is Ahlfors regular (or $n$-Ahlfors regular) if $\HH^n|_E$ is Ahlfors regular.

A measure $\mu$ in $\R^d$ is called {\em uniformly $n$-rectifiable} if it is $n$-Ahlfors regular and
there exist constants $\theta, M >0$ such that for all $x \in \supp\mu$ and all $0<r\leq \diam(\supp\mu)$ 
there is a Lipschitz mapping $g$ from the ball $B_n(0,r)$ in $\R^{n}$ to $\R^d$ with $\text{Lip}(g) \leq M$ such that
$$
\mu(B(x,r)\cap g(B_{n}(0,r)))\geq \theta r^{n}.$$
A set $E\subset\R^d$ is called uniformly  $n$-rectifiable if the measure $\HH^n|_E$ is uniformly $n$-rectifiable.
Recall that $E$ is called $n$-rectifiable if there are Lipschitz maps $g_i:\R^n\to\R^d$
such that 
$$\HH^n\Big(E \setminus \textstyle\bigcup_i g_i(\R^n)\Big)=0.$$
It is easy to check that if $E$ is uniformly  $n$-rectifiable, then it is $n$-rectifiable.
\vv
% **************************************************************

\subsection{NTA and Reifenberg flat domains}
\label{secprelim-nta}

 Given $\Omega\subset \mathbb{R}^{n+1}$ and $C\geq 2$, 
 %$R>0$,
we say that $\Omega$ satisfies the {\it $C$-Harnack chain condition} if 
%there is a constant $R>0$ such that
for every $\rho>0$, $k\geq 1$, and every pair of points
$x,y \in \Omega$ 
with $\dist(x,\partial\Omega),\,\dist(y,\partial\Omega) \geq\rho$ and 
$|x-y|<2^k\rho$, there is a chain of
open balls
$B_1,\dots,B_m \subset \Omega$, $m\leq C k$,
with $x\in B_1,\, y\in B_m,$ $B_k\cap B_{k+1}\neq \varnothing$
and $C^{-1}\diam (B_k) \leq \dist (B_k,\partial\Omega)\leq C\,\diam (B_k).$  The chain of balls is called
a {\it Harnack chain}.

For $C\geq 2$, $\Omega$ is a {\it $C$-corkscrew set} if for all $\xi\in \partial\Omega$ and $r\in(0,R)$ there is a ball of radius $r/C$ contained in $B(\xi,r)\cap \Omega$. Finally, we say that $\Omega$ is {\it $C$-non-tangentially accessible (or $C$-NTA, or just NTA)} if it satisfies the Harnack chain condition and both $\Omega$ and 
$(\overline{\Omega})^{c}$ 
are $C$-corkscrew sets.   Also, $\Omega$ is
{\it two-sided $C$-NTA} if both $\Omega$ and $(\overline{\Omega})^{c}$ are $C$-NTA. 
A chord-arc domain is an NTA domain $\Omega\subset\R^{n+1}$ whose boundary is $n$-Ahlfors regular. NTA domains were introduced by Jerison and Kenig in \cite{Jerison-Kenig}. In this type of domains, the authors showed that harmonic measure is doubling and it satisfies other remarkable properties, such as the following change of pole formula.

\begin{theorem}[{\cite[Lemma 4.11]{Jerison-Kenig}}]\label{teounif}
Let $n\geq 1$, $\Omega$ be a $C$-NTA open set in $\R^{n+1}$ and let $B$ be a ball centered in $\partial\Omega$.
Let $p_1,p_2\in\Omega$ be such that $\dist(p_i,B\cap \partial\Omega)\geq c_0^{-1}\,r(B)$ for $i=1,2$.
Then, for any Borel set $E\subset B\cap\partial\Omega$,
\begin{equation}\label{eqchangepole}
\frac{\omega^{p_1}(E)}{\omega^{p_1}(B)}\approx \frac{\omega^{p_2}(E)}{\omega^{p_2}(B)},\end{equation}
with the implicit constant depending only on $n$, $c_0$ and $C$. 
\end{theorem}
\vv

Given a set $E\subset\R^{n+1}$, $x\in \mathbb{R}^{n+1}$, $r>0$, and a hyperplane $P$, we set
\begin{equation}\label{eqDE}
D_{E}(x,r,P)=r^{-1}\max\left\{\sup_{y\in E\cap B(x,r)}\dist(y,P), \sup_{y\in P\cap B(x,r)}\dist(y,E)\right\}.
\end{equation}
We also define
\begin{equation}\label{eqDE2}
D_{E}(x,r)=\inf_{P}D_{E}(x,r,P)
\end{equation}
where the infimum is taken over all hyperplanes $P$. Given $\delta, R>0$, a set $E\subset\R^{n+1}$ is {\em $(\delta,R)$-Reifenberg flat} (or just $\delta$-Reifenberg flat) if $D_{E}(x,r)<\delta$ for all $x\in E$ and $0<r\leq R$, and it is {\it vanishing Reifenberg flat} if 
\[
\lim_{r\rightarrow 0} \sup_{x\in E} D_{E}(x,r)=0.\]
 %A domain $\Omega$ is {\it $\delta$-Reifenberg flat} if $\d\Omega$ is $\delta$-Reifenberg flat, and vanishing Reifenberg flat is defined similarly. 

%\begin{definition}\label{defreif}
Let $\Omega\subset \R^{n+1}$ be an open set, and let $0<\delta<1/2$. We say that $\Omega$
is a $(\delta,R)$-Reifenberg flat domain (or just $\delta$-Reifenberg flat) if it satisfies the following conditions:
\begin{itemize}
\item[(a)] $\partial\Omega$ is $(\delta,R)$-Reifenberg flat.

\item[(b)] For every $x\in\partial \Omega$ and $0<r\leq R$, denote by $P(x,r)$ an $n$-plane that minimizes $D_{E}(x,r)$. Then one of the connected components of 
$$B(x,r)\cap \bigl\{x\in\R^{n+1}:\dist(x,P(x,r))\geq 2\delta\,r\bigr\}$$
is contained in $\Omega$ and the other is contained in $\R^{n+1}\setminus\Omega$.
\end{itemize}
If, additionally,  $\partial\Omega$ is vanishing Reifenberg flat, then $\Omega$ is said to be vanishing Reifenberg flat, too.
It is well known that if $\Omega$ is a $\delta$-Reifenberg flat domain, with $\delta$ small enough,
then it is also an NTA domain (see \cite{Kenig-Toro-duke}).
%\end{definition}

{
Given two NTA domains $\Omega^+\subset \R^{n+1}$ and $\Omega^-=\R^{n+1}\setminus \overline{\Omega^+}$, 
we say that $\Omega^+$ and $\Omega^-$ have  joint big pieces of chord-arc subdomains if for any ball $B$ centered in $\partial\Omega^+$ with radius at most $\diam (\partial\Omega^+)$ there are two chord-arc domains
$\Omega_{B}^s\subset \Omega^s$, with $s=+,-$, such that $\HH^n(\partial\Omega_B^+\cap 
\partial\Omega_B^-\cap B)\gtrsim r(B)^n$. }

 \vv

% ********************************************************************************
% ********************************************************************************

\subsection{The space $\vmo$}

Given a Radon measure $\mu$ in $\R^{n+1}$, $f\in L^1_{loc}(\mu)$, and $A\subset \R^{n+1}$, we write
$$m_{\mu,A}(f) = \avint_A f\,d\mu=
\frac1{\mu(A)}\int_A f\,d\mu.$$
Assume $\mu$ to be doubling.
%Let $\sigma$ be a doubling Radon measure in $\mathbb{R}^{n+1}$ and $f$ a locally integrable function with respect to $\sigma$. 
We say that $f\in \vmo(\mu)$ if
\begin{equation}\label{defvmo}
\lim_{r\rightarrow 0}\sup_{x\in \supp \mu}\,\, \avint_{B(x,r)} \left| f- m_{\mu,B(x,r)}(f)\right|^{2}d\mu =0.
\end{equation}
It is well known  that the space $\vmo(\mu)$  coincides with the closure of the set of bounded uniformly continuous functions on $\supp \mu$ in the BMO norm.

\vvv
% ********************************************************************************
% ********************************************************************************

\section{Construction of the domain}

We will construct a domain $\Omega=\Omega^+\subset\R^2$ by a limiting procedure. First we let $\Omega_0$ be the interior of a regular polygon with 100 sides, say, with side length equal to $1$. Given $\Omega_n$, to construct $\Omega_{n+1}$, we assume that $\Omega_n$ is a Jordan domain such that its boundary is a piecewise linear curve made up
of $m_n$ closed segments $L_1^n,\ldots,L_{m_n}^n$ with equal length $\ell_n$. We suppose that they are ordered clockwise, so that there are points $z_1^n,\ldots,z_{m_n}^n\in\partial\Omega_n$ such that 
$L_j^n=[z_j^n,z_{j+1}^n]$ for each $j=1,\ldots,m_n$, with $z_{m_n+1}^n=z_1^n$. We assume also that, if $B_j^n$ is a closed ball centered in the mid-point of $L_j^n$ with radius $\ell_n/2$, then one component of $B_j^n\setminus L_j^n$ is contained in $\Omega_n$ and the other in $\R^2\setminus \Omega_n$.

The first step to construct $\Omega_{n+1}$ consists in replacing each segment $L_j^n$ by a piecewise linear Jordan arc $\Gamma_j^{n}$ with the same end-points as $L_j^n$.
\vv

\subsection{Construction of $\Gamma_j^{n}$}

To shorten notation, denote $L=L_j^n$, and let $N(L)$ be the outer unit normal of $\Omega_n$ at $L$. 
Denote by $\gamma_L$ the angle between $N(L)$ and the horizontal vector $(1,0)$. We adopt the convention that
$\gamma_L\in (-\pi,\pi]$. %If $\gamma_L=0$, we let $\Gamma_j^n=L$. %We also write $F(L) = L$.
We will define $\Gamma_j^n$ by a snowflake type construction. First we denote
$$\alpha=\alpha_L=\frac{|\gamma_L|}{M_{n}},$$
where $M_{n}$ is a big integer (say $M_{n}\geq 20$) which will be chosen inductively later. For the moment, let us 
say that $M_1=20$ and that $\{M_n\}_n$ will be a monotone sequence of natural numbers tending to $\infty$.
Now we start the usual construction of an arc of the Koch snowflake with angle $\alpha$ instead of $\pi/3$. To this end,
we denote $\ell=\HH^1(L)$ (so $\ell=\ell_n$) and
we replace $L$ by four segments $S_1,\ldots,S_4$ with equal length
$$s:= \frac1{2(1+\cos\alpha)}\,\ell,$$
as in Figure \ref{figure1}. That is, in the case when $L=[(0,0),(\ell,0)]$ is a horizontal segment and $N(L)=(0,1)$, we 
consider the points (in complex notation)
$$y_0=0,\quad y_1=s,\quad y_2 = \frac\ell2 + i\,s\,\sin\alpha,\quad y_3 = \ell-s,\quad y_4=\ell,$$
and we let $S_i = [y_{i-1},y_i]$.

\usetikzlibrary {turtle}

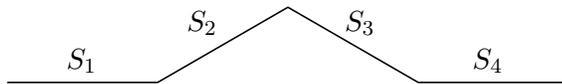
\begin{figure}
\begin{center}
\vvv
\begin{tikzpicture}
\draw  (10mm,3mm) node {$S_1$};
 \draw  (26mm,8mm) node {$S_2$};
 \draw  (47mm,8mm) node {$S_3$};
 \draw  (64mm,3mm) node {$S_4$};

\tikz[turtle/distance=20mm]
 \draw [semithick,turtle={home,rt,
  fd, %
  lt=30,
  fd, %
  rt=60,
  fd, %
  lt=30,
  fd %
 }];

 \end{tikzpicture}
\caption{The segments $S_1$, $S_2$, $S_3$, $S_4$, with $\alpha=\pi/6$.}\label{figure1}
\end{center}
\end{figure}
\vvv

For an arbitrary segment $L$, we define $S_1,\ldots,S_4$ so that after a suitable translation and rotation we are in the preceding situation.
We denote by $F(L)$ the curve generated in this way. That is, 
$$F(L) = S_1\cup S_2\cup S_3\cup S_4.$$
We also denote by $N(S_i)$ the unit normal to each vector $S_i$, so that the angle between $N(L)$ and $N(S_i)$ is at most $\alpha$. In other words, $N(S_i)$ is the outer unit normal at $S_i$ of the domain enclosed by the curve obtained by replacing $L$ by $F(L)$ in $\partial\Omega_n$.

In the first iteration, we let $\Gamma_1(L)= F(L)$. To construct, $\Gamma_2(L)$ we iterate the construction in the usual way: we let 
$$\Gamma_2(L) = F(S_1)\cup F(S_2)\cup F(S_3) \cup F(S_4),$$
with $F(S_i)$ defined in the same way as $F(L)$, with $L$ replaced by $S_i$.
We denote the four segments which compose $F(S_i)$ by $S_{i,1},\ldots,S_{i,4}.$ To construct $\Gamma_3(L)$ we proceed similarly. 

However, we introduce a special rule in the iteration of the next curves $\Gamma_{k+1}(L)$:
if one of the segments $S_{i_1,\ldots,i_k}$ of $\Gamma_k(L)$ is vertical and the associated outer normal $N(S_{i_1,\ldots,i_k})$ is the vector $(1,0)$, then in the construction of $\Gamma_{k+1}(L)$ we keep 
$S_{i_1,\ldots,i_k}$ unchanged. 
That is, for a segment $S_{i_1,\ldots,i_k}$ we let
\begin{equation}\label{eqftilde}
\wt F(S_{i_1,\ldots,i_k}) = \left\{\begin{array}{ll} F(S_{i_1,\ldots,i_k}) & \quad \text{if $N(S_{i_1,\ldots,i_k})\neq (1,0),$}\\
S_{i_1,\ldots,i_k}& \quad \text{if $N(S_{i_1,\ldots,i_k}) = (1,0).$}
\end{array}\right.
\end{equation}
In the latter case, we let $S_{i_1,\ldots,i_k,1},\cdots,S_{i_1,\ldots,i_k,4}$ be the four closed segments obtained 
by splitting $S_{i_1,\ldots,i_k}$ into four segments with disjoint interiors and the same length, and we let
$N(S_{i_1,\ldots,i_j,h})= (1,0)$ for $h=1,2,3,4$.
Then we let
$$\Gamma_{k+1}(L) = \bigcup_{1\leq i_1,\ldots,i_k\leq 4} \wt F(S_{i_1,\ldots,i_k}).$$
Notice that, by the definition of $\alpha$, the first appearance of a vertical segment 
$S_{i_1,\ldots,i_k}$ such that $N(S_{i_1,\ldots,i_k}) = (1,0)$ when $\alpha\neq0$ occurs for $k=M_{n}$, and so the last definition is coherent with the construction described for the first curves $\Gamma_1(L),\Gamma_2(L),\Gamma_3(L)$.

We iterate this construction $M_n^2\equiv(M_{n})^2$ times, and we let
$$\Gamma_j^{n} = \Gamma_{M_{n}^2}(L).$$ 
Observe that $\Gamma_j^{n}$ is made up of $4^{M_n^2}$ segments $S_{i_1,\ldots,i_{M_n^2}}$. See Figure \ref{figure2} (where we took $M_n=2$ for an easy viewing of the vertical segments with associated normal $(1,0)$).

\begin{figure}
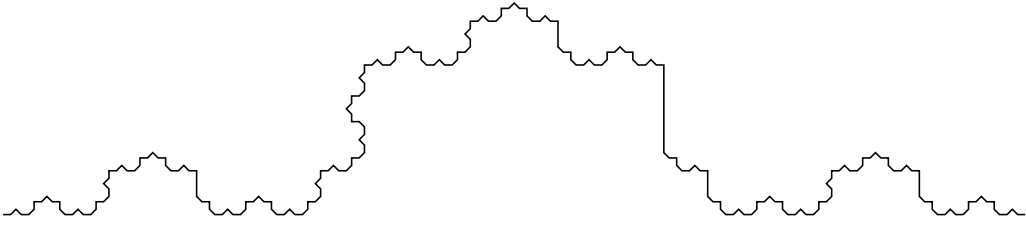

\begin{center}
\vvv
%\usetikzlibrary {turtle}
\tikz[turtle/distance=1mm]
  \draw [semithick,turtle={home,rt,
   fd,lt=45,fd,rt=90,fd,lt=45,fd,  % 
  lt=45,
   fd,lt=45,fd,rt=90,fd,lt=45,fd,  %
  rt=90,
   fd,lt=45,fd,rt=90,fd,lt=45,fd,  %
  lt=45,
   fd,lt=45,fd,rt=90,fd,lt=45,fd,  %
  lt=45,
   fd,lt=45,fd,rt=90,fd,lt=45,fd,  %
  lt=45,
   fd,lt=45,fd,rt=90,fd,lt=45,fd,  %
  rt=90,
   fd,lt=45,fd,rt=90,fd,lt=45,fd,  %
  lt=45,
   fd,lt=45,fd,rt=90,fd,lt=45,fd,  %
  rt=90,
   fd,lt=45,fd,rt=90,fd,lt=45,fd,  %
  lt=45,
   fd,lt=45,fd,rt=90,fd,lt=45,fd,  %
  rt=90,
  fd=3.414mm, %%%
  lt=45,
   fd,lt=45,fd,rt=90,fd,lt=45,fd,  %
  lt=45,
   fd,lt=45,fd,rt=90,fd,lt=45,fd,  %
  lt=45,
   fd,lt=45,fd,rt=90,fd,lt=45,fd,  %
  rt=90,
   fd,lt=45,fd,rt=90,fd,lt=45,fd,  %
  lt=45,
   fd,lt=45,fd,rt=90,fd,lt=45,fd,  %
  lt=45,
   fd,lt=45,fd,rt=90,fd,lt=45,fd,  %
  lt=45,
   fd,lt=45,fd,rt=90,fd,lt=45,fd,  %
  rt=90,
   fd,lt=45,fd,rt=90,fd,lt=45,fd,  %
  lt=45,
   fd,lt=45,fd,rt=90,fd,lt=45,fd,  %
  lt=45,
   fd,lt=45,fd,rt=90,fd,lt=45,fd,  %
  lt=45,
   fd,lt=45,fd,rt=90,fd,lt=45,fd,  %
  rt=90,
   fd,lt=45,fd,rt=90,fd,lt=45,fd,  %
  lt=45,
   fd,lt=45,fd,rt=90,fd,lt=45,fd,  %
  rt=90,
   fd,lt=45,fd,rt=90,fd,lt=45,fd,  %
  lt=45,
   fd,lt=45,fd,rt=90,fd,lt=45,fd,  %
  rt=90,
   fd,lt=45,fd,rt=90,fd,lt=45,fd,  %
  lt=45,
   fd,lt=45,fd,rt=90,fd,lt=45,fd,  %
  lt=45,
   fd,lt=45,fd,rt=90,fd,lt=45,fd,  %
  lt=45,
   fd,lt=45,fd,rt=90,fd,lt=45,fd,  %
  rt=90,
   fd,lt=45,fd,rt=90,fd,lt=45,fd,  %
  lt=45,
   fd,lt=45,fd,rt=90,fd,lt=45,fd,  %
  rt=90,
   fd,lt=45,fd,rt=90,fd,lt=45,fd,  %
  lt=45,
   fd,lt=45,fd,rt=90,fd,lt=45,fd,  %
  rt=90,
  fd=3.414mm, %%%
  lt=45,
   fd,lt=45,fd,rt=90,fd,lt=45,fd,  %
  lt=45,
   fd,lt=45,fd,rt=90,fd,lt=45,fd,  %
  lt=45,
   fd,lt=45,fd,rt=90,fd,lt=45,fd,  %
  rt=90,
   fd,lt=45,fd,rt=90,fd,lt=45,fd,  %
  lt=45,
   fd,lt=45,fd,rt=90,fd,lt=45,fd,  %
  rt=90,
  fd=11.657mm,%%%%%%
  lt=45,
   fd,lt=45,fd,rt=90,fd,lt=45,fd,  %
  lt=45,
   fd,lt=45,fd,rt=90,fd,lt=45,fd,  %
  rt=90,
  fd=3.414mm, %%%
  lt=45,
   fd,lt=45,fd,rt=90,fd,lt=45,fd,  %
  lt=45,
   fd,lt=45,fd,rt=90,fd,lt=45,fd,  %
  lt=45,
   fd,lt=45,fd,rt=90,fd,lt=45,fd,  %
  rt=90,
   fd,lt=45,fd,rt=90,fd,lt=45,fd,  %
  lt=45,
   fd,lt=45,fd,rt=90,fd,lt=45,fd,  %
  lt=45,
   fd,lt=45,fd,rt=90,fd,lt=45,fd,  %
  lt=45,
   fd,lt=45,fd,rt=90,fd,lt=45,fd,  %
  rt=90,
   fd,lt=45,fd,rt=90,fd,lt=45,fd,  %
  lt=45,
   fd,lt=45,fd,rt=90,fd,lt=45,fd,  %
  rt=90,
   fd,lt=45,fd,rt=90,fd,lt=45,fd,  %
  lt=45,
   fd,lt=45,fd,rt=90,fd,lt=45,fd,  %
  rt=90,
  fd=3.414mm, %%%
  lt=45,
   fd,lt=45,fd,rt=90,fd,lt=45,fd,  %
  lt=45,
   fd,lt=45,fd,rt=90,fd,lt=45,fd,  %
  lt=45,
   fd,lt=45,fd,rt=90,fd,lt=45,fd,  %
  rt=90,
   fd,lt=45,fd,rt=90,fd,lt=45,fd,  %
  lt=45,
  fd,lt=45,fd,rt=90,fd,lt=45,fd
  }];
\end{center}
\vvv

\caption{The curve $\Gamma_j^n$ generated by a horizontal segment $L_j^n$ with parameters $N(L_j^n)=(1,0)$, $\alpha=\pi/4$, and $M_n=2$.}
\label{figure2}
\end{figure}

\vv
\subsection{The chord-arc property of $\Gamma_j^n$}

Notice that for $1\leq k\leq M_n^2$, 
\begin{equation}\label{eq2ineq*}
\frac1{4}\,\HH^1(S_{i_1,\ldots,i_{k-1}})\leq \HH^1(S_{i_1,\ldots,i_{k}})
\leq \frac1{2(1+\cos\alpha)}\,\HH^1(S_{i_1,\ldots,i_{k-1}}).
\end{equation}
Thus, 
\begin{equation}\label{eq2ineq*4}
4^{-k}\ell\leq 
\HH^1(S_{i_1,\ldots,i_{k}})\leq \left(\frac1{2(1+\cos\alpha)}\right)^{k}\,\ell 
\quad\mbox{ for $1\leq k\leq M_n^2$.}
\end{equation}
We can also write  the second inequality as follows:
$$\HH^1(S_{i_1,\ldots,i_{k}})\leq 4^{-k} \left(\frac{2}{1+\cos\alpha}\right)^{k}\,\ell
= 4^{-k}\left(1+ \frac{1-\cos\alpha}{1+\cos\alpha}\right)^k\,\ell.$$
Next we use the fact that, by the definition of $\alpha$,
$$1+ \frac{1-\cos\alpha}{1+\cos\alpha} \leq1 + (1-\cos\alpha) \leq 1 +C\alpha^2\leq 1 + \frac{C}{M_n^2}.$$
Thus, using also that $k\leq M_n^2$,
\begin{equation}\label{eq2ineq*5}
\HH^1(S_{i_1,\ldots,i_{k}})\leq 4^{-k}\left(1+ \frac{1-\cos\alpha}{1+\cos\alpha}\right)^{M_n^2}\,\ell\leq
4^{-k}\left(1+ \frac{C}{M_n^2}\right)^{M_n^2}\,\ell \lesssim 4^{-k}\ell.
\end{equation}
Consequently,
 since $\Gamma_j^n$ consists of $4^{M_n^2}$ segments of the form $S_{i_1,\ldots,i_{M_n^2}}$, we derive
\begin{equation}\label{eqcorarc10}
\HH^1(\Gamma_j^n)\leq 4^{M_n^2}\,\sup_{i_1,\ldots,i_{M_n^2}} \HH^1(S_{i_1,\ldots,i_{M_n^2}}) \lesssim 
4^{M_n^2}\,4^{-M_n^2}\,\ell = \ell.
\end{equation}

The above calculation suggests that $\Gamma_j^n$ is an Ahlfors regular curve. In the next lemma we show that this is the case. In fact,
we will prove that this is a bilipschitz image of the segment $L$, thus a chord-arc curve, and in particular an Ahlfors regular curve.

\begin{lemma}\label{lembilip}
Suppose that $\alpha$ is small enough.
For each $k=1,\ldots,M_n^2$, the curve $\Gamma_k(L)$ is a bilipschitz image of the segment $L$. 
That is, there exist a bijective map $\vphi_k:L\to \Gamma_k(L)$ and a constant $C_k$ such that for all 
$x,y\in L$,
\begin{equation}\label{eqbilip}
C_k^{-1}|x-y|\leq |\vphi_k(x) - \vphi_k(y)|\leq C_k\,|x-y|.
\end{equation}
Further the constant $C_k$ is bounded above by some absolute constant. 
\end{lemma}

Of course, in particular the lemma implies that  $\Gamma_j^n$ is an Ahlfors regular curve, with the
Ahlfors regularity constant bounded by an absolute constant.

%Remark that in the lemma we assume $\alpha$ small enough, which suffices for our purposes. However, very likely, with some additional effort one should be able to prove the lemma for  $\alpha\leq \pi/3$, for example.

\begin{proof}
We argue by induction. The case $k=1$ is clear. Fix now $k\in [2,M_n^2]$, and suppose that $\Gamma_j(L)$ is a chord-arc curve and that
\rf{eqbilip} holds for $1\leq j\leq k-1$. Denote $L=\Gamma_0(L)$. 
For $0\leq j\leq k-1$, we define a map $\Pi_j:\Gamma_j(L) \to \Gamma_{j+1}(L)$ as follows. 
In the case $j=0$, we let $\Pi_0$ be equal to the projection from $L$ to $\Gamma_1(L)$ orthogonal to $L$.
In the case $j\geq1$, recall that
$$\Gamma_j(L) =\bigcup_{1\leq i_1,\ldots,i_j\leq 4} S_{i_1,\ldots,i_j}.$$
Then, on each segment $S_{i_1,\ldots,i_j}$ we let $\Pi_j|_{S_{i_1,\ldots,i_j}}$ be equal to the projection from 
$S_{i_1,\ldots,i_j}$ to $\wt F(S_{i_1,\ldots,i_j})$ orthogonal to $S_{i_1,\ldots,i_j}$ (recall that $\wt F(S_{i_1,\ldots,i_j})$ is defined in \rf{eqftilde}).
Notice that $\Gamma_j$ does not have self-intersections, since it is a chord-arc curve, by assumption. So 
the definition of $\Pi_j$ is correct because each $z\in\Gamma_j(L)$ belongs at most to one segment $S_{i_1,\ldots,i_j}$,
except in the case when $z$ is a common end point of two adjacent segments of the form $S_{i_1,\ldots,i_j}$ (in which case $\Pi_j(z)=z$).

We define
$$\vphi_k= \Pi_{k-1}\circ\cdots \Pi_1\circ\Pi_0.$$
By construction it is clear that $\vphi_k:L\to \Gamma_k(L)$ is continuous, piecewise affine, and surjective. To check that \rf{eqbilip} holds for all $x,y\in L$, notice first that, by the definition of $\vphi_k$ and the construction of the $\Gamma_k(L)$, we can split $L$ into $4^k$ segments 
$I_{i_1,\ldots,i_k}$, with $1\leq i_1,\ldots,i_k\leq 4$, with disjoint interiors, with $\HH^1(I_{i_1,\ldots,i_k}) = 4^{-k}\ell$, such that
$$\vphi_k(I_{i_1,\ldots,i_k}) = S_{i_1,\ldots,i_k}\quad \mbox{ for $1\leq i_1,\ldots,i_k\leq 4$,}$$
and with $\vphi_k|_{I_{i_1,\ldots,i_k}}$ being affine in $I_{i_1,\ldots,i_k}$.
Then, recalling that, by \rf{eq2ineq*4} and \rf{eq2ineq*5}, 
$$\HH^1(S_{i_1,\ldots,i_k}) \approx  \HH^1(I_{i_1,\ldots,i_k})=4^{-k}\ell,$$
we deduce $|\nabla\vphi_k|\approx1$ in $I_{i_1,\ldots,i_k}$. By connectivity, this implies that $\vphi_k$ is Lipschitz on $L$, with
${\rm Lip}(\vphi_k)\lesssim1$. Remark that almost the same argument shows that
$\Pi_{m}\circ\cdots \Pi_1\circ\Pi_0$ is affine in $I_{i_1,\ldots,i_m}$, for $0\leq m \leq k-1$ and
\begin{equation}\label{eausd5}
|\nabla(\Pi_{m}\circ\cdots \Pi_1\circ\Pi_0)|\approx1.
\end{equation}

It remains to check that
\begin{equation}\label{eqxygeq}
|\vphi_k(x) - \vphi_k(y)|\gtrsim |x-y|\quad \mbox{ for all $x,y\in L$.}
\end{equation}
To this end,  for 
$1\leq j\leq k$, we denote $$x_{j}= \Pi_{j-1}\circ\cdots \Pi_1\circ\Pi_0(x),\quad
y_j= \Pi_{j-1}\circ\cdots \Pi_1\circ\Pi_0(y),
$$
so that, in particular, $\vphi_k(x)=x_k$ and $\vphi_k(y)=y_k$. Notice also that $x_j,y_j\in \Gamma_j(L)$.
Observe also that, by the definition of $\Pi_j$ and $F(S_{i_1,\ldots,i_j})$ and \rf{eq2ineq*5} with $k=j$,
one easily gets 
\begin{equation}\label{eqfac837}
|x_{j+1}-x_j|\leq \HH^1(S_{i_1,\ldots,i_j})\,\sin\alpha \leq C\,4^{-j}\,\ell\,\sin\alpha.
\end{equation}

Let $h\geq0$ be the integer such that
\begin{equation}\label{eqhhh444}
4^{-h-1}\ell\leq |x-y|< 4^{-h}\ell.
\end{equation}
 Suppose first that $h\leq k$. Then either exists a segment $I_{i_1,\ldots,i_h}\subset L$ such that
$$x,y \in I_{i_1,\ldots,i_h},$$
or there are two adjacent segments (i.e, with a common end-point) $I_{i_1,\ldots,i_h},\, I_{j_1,\ldots,j_h}\subset L$ such that
$$x \in I_{i_1,\ldots,i_h},\qquad y \in I_{j_1,\ldots,j_h}.$$
In the first case, since $\Pi_{h-1}\circ\cdots \Pi_1\circ\Pi_0$ is affine in $I_{i_1,\ldots,i_{h}}$, by \rf{eausd5} we have
\begin{equation}\label{eqcase1**}
|x-y| \approx |x_h-y_h|.
\end{equation}
In the second case, this also holds. Indeed, let $z$ be the common end-point of $I_{i_1,\ldots,i_h}$ and $I_{j_1,\ldots,j_h}$, and let 
$z_h= \Pi_{h-1}\circ\cdots \Pi_1\circ\Pi_0(z)$. Then,
\begin{equation}\label{eqcase2**}
|x-y| = |x-z| + |z-y| \approx |x_h-z_h| + |z_h-y_h| \approx |x_h-y_h|,
\end{equation}
where we applied \rf{eqcase1**} 
twice in the first estimate to show that $|x-z|\approx |x_h-z_h|$ and $|z-y| \approx  |z_h-y_h|$.
 In the last estimate in \rf{eqcase2**} we used the fact that since the two segments $I_{i_1,\ldots,i_h}$ and $I_{j_1,\ldots,j_h}$ are adjacent in $L$, then $S_{i_1,\ldots,i_h}$ and $S_{j_1,\ldots,j_h}$ are adjacent segments in $\Gamma_h(L)$ and they form an angle
either equal to $\pi-\alpha$ or $\pi$, with $0<\alpha\ll1$, by the construction above.

The estimate \rf{eqcase1**}, together with \rf{eqfac837}, gives
\begin{align*}
|\vphi_k(x) - \vphi_k(y)| & = |x_k-y_k|\geq |x_h-y_h| - \sum_{m=h}^{k-1} |x_m-x_{m+1}| - \sum_{m=h}^{k-1} |y_m-y_{m+1}|\\
&\geq c\,|x-y| - C\,4^{-h}\,\ell\,\sin\alpha \gtrsim |x-y|,
\end{align*}
assuming $\alpha$ small enough in the last inequality.

In the case when the integer $h$ in \rf{eqhhh444} is larger than $k$, then either exists a segment $I_{i_1,\ldots,i_k}$ such that
$x,y \in I_{i_1,\ldots,i_k},$
or there are two adjacent segments $I_{i_1,\ldots,i_k}$, $I_{j_1,\ldots,j_k}$, such that
$x \in I_{i_1,\ldots,i_k}$ and $y \in I_{j_1,\ldots,j_k}$. Then, we deduce that either \rf{eqcase1**} or \rf{eqcase2**} hold
replacing $x_h,y_h,z_h$ by $x_k,y_k,z_k$, by the same arguments above. So in any case, \rf{eqxygeq} holds, which concludes the proof of the
lemma.
\end{proof}
\vv

%From the fact that $\Gamma_k(L)$ is a chord-arc for $1\leq k\ eq M_n^2$ and the fact that the angle between adjacent segments %$S_{i_1,\ldots,i_k}$ and $S_{j_1,\ldots,j_k}$ in $\Gamma_k(L)$ form an angle close to a plane angle (in fact, either equal to $\pi$ or $\pi-\alpha$), 

The curves $\Gamma_j^n$ also satisfy the following:

\begin{lemma}\label{lemflat}
For every $\xi\in\Gamma_j^n$ and $r>0$, there exists a line $L_{\xi,r}$ such that
\begin{equation}\label{eqgjnalf}
\Gamma_j^n\cap B(\xi,r)\subset \UU_{c\alpha r}(L_{\xi,r})\cap B(\xi,r).
\end{equation}
\end{lemma}

In the lemma, $\UU_{s}(A)$ stands for the $s$-neighborhood of $A$.

\begin{proof}
We will use the same notation as in the proof of Lemma \ref{lembilip}.

We may assume that $r\leq \ell$.
Let $k\geq0$ be the integer such that
%\begin{equation}\label{eqhhh444*}
$4^{-k-1}\ell\leq r< 4^{-k}\ell$.
%\end{equation}
 Suppose first that $k\leq M_n^2$.  
Let $x=(\vphi_{M_n^2})^{–1}(\xi)$ and let $x_k=\vphi_k(x)$. 
Let $S_{i_1,\ldots,i_k}$ be such that $x_k\in S_{i_1,\ldots,i_k}$ and let $L_{\xi,r}$ be the line supporting $S_{i_1,\ldots,i_k}$. As in the proof of Lemma \ref{lembilip}, denote $I_{i_1,\ldots,i_k} =\vphi_k^{-1}(S_{i_1,\ldots,i_k})$.
  For some constant $A_0>10$ to be fixed below, let
$\mathcal I$ the family of segments $I_{j_1,\ldots,j_k}$ such that
$$\dist(I_{i_1,\ldots,i_k},I_{j_1,\ldots,j_k})\leq A_0\,4^{-k}\ell,$$
and let 
$\mathcal I'= \{\vphi_k(J):J\in \mathcal I\}.$
Since the angle between $S_{i_1,\ldots,i_k}$ and any segment $S\in\mathcal I'$ is between  $\pi-C(A_0)\alpha$ and $\pi$,
it follows that 
$S\subset \UU_{C(A_0)\alpha r}(L_{\xi,r})$ for any $S\in\mathcal I'$. On the other hand, 
for $S\not\in\mathcal I'$ and $A_0$ large enough, we have $S\cap B(x_k,2r)=\varnothing$ by the chord-arc property of $\Gamma_k(L)$.
Thus 
\begin{equation}\label{eqgammah}
\Gamma_k(L) \cap B(x_k,2r)\subset \UU_{C(A_0)\alpha r}(L_{\xi,r}).
\end{equation}
From \rf{eqfac837}, it easily follows that 
$\Gamma_j^n= \Gamma_{M_n^2}(L)\subset \UU_{c\alpha 4^{-k}}(\Gamma_k(L)).$
The lemma is deduced from this fact and \rf{eqgammah}.
The proof for the case $h>M_n^2$ is similar, and even simpler.
\end{proof}
\vv

For a point $z\in\R^2$, a line $\Lambda\subset \R^2$, and $\beta>0$, we denote by $X(z,\Lambda,\beta)$ the cone
$$X(z,\Lambda,\beta):= \{y\in\R^2:\dist(z,\Lambda)\leq \beta\,|y-z|\}.$$
Another useful property of $\Gamma_j^n$ is the following: if $x_{L_j^n},y_{L_j^n}$ are the end-points of $L^n_j$ and $\Lambda_j^n$ is the line supporting $L_j^N$, then 
\begin{equation}\label{eqgammah23}
\Gamma_j^n\subset X(x_j^n,\Lambda_j^n,C\alpha) \cap X(y_j^n,\Lambda_j^n,C\alpha).
\end{equation}
This can deduced by a quick inspection of the construction of $\Gamma_j^n$. 
We leave the details for the reader.

\vv
\subsection{The abundance of vertical segments in the curve $\Gamma_j^n$}

Denote by $\mathcal V(\Gamma_j^n)$ the subfamily of the segments $S_{i_1,\ldots,i_{M_n^2}}$, with $1\leq i_1,\ldots,i_{M_n^2}\leq4$, which are vertical and such that $N(S_{i_1,\ldots,i_{M_n^2}})=(1,0)$.
Before going on with the construction of $\Omega_{n+1}$, we will prove the abundance of that type of segments. This will be the key property that we will use below to show that the outer unit normal to $\Omega$ equals $(1,0)$ $\omega^+$-a.e.
The precise result we need is the following:

\begin{lemma}\label{lemvertical}
There exists an absolute constant $c_1>0$ such that, for any choice of $M_n$ large enough,
\begin{equation}\label{eqsumguai}
\sum_{S\in \mathcal V(\Gamma_j^n)} \HH^1(S) \geq c_1\,\HH^1(\Gamma_j^n).
\end{equation}
\end{lemma}

\begin{proof}
We will use a probabilistic argument. We can assume that $\alpha\neq0$, because otherwise all the segments
$S_{i_1,\ldots,i_{M_n^2}}$ are vertical and their associated unit normal is $(1,0)$.
Consider the set of codings of the segments $S_{i_1,\ldots,i_{M_n^2}}$, that is,
$I_n := \{1,2,3,4\}^{M_n^2}$. Let $\mu$ be the uniform probability measure on $\{1,2,3,4\}$ (so that
$\mu(\{1\}) =\cdots=\mu(\{4\})=1/4$), and let $\mu_{I_n}= \mu\times\cdots\times\mu$ be the product measure ($M_n^2$ times) of $\mu$ on $I_n$. Consider the function 
$g:\{1,2,3,4\}\to\R$ defined by
$$g(1)=g(4)=0,\quad g(2)=\alpha,\quad g(3)=-\alpha,$$
and consider the random variables $X_1,\ldots,X_{M_n^2}$ on $I_n$ defined by
$$X_j\big((i_1,\ldots,i_{M_n^2})\big) = g(i_j).$$
It is immediate to check that the variables $X_1,\ldots,X_{M_n^2}$ are independent and identically distributed, and they have zero mean and variance $\sigma^2= \alpha^2/2.$

Notice that, for $i\in I_n$ with $S_i\not\in \mathcal V(\Gamma_j^n)$, we have that
$\Sigma_{M_n^2}(i):= X_1(i)+\ldots + X_{M_n^2}(i)$ is the angle that the segment $S_i$ has rotated with respect the initial segment\footnote{In fact, $\Sigma_{M_n^2}(i)$ coincides also with the angle that any segment from $\{S_i\}_{i\in I_n}$ would have rotated if we had not applied the special rule about the vertical segments in the construction of $\Gamma_j^n$.} $L_j^n$.
Suppose for simplicity that the angle $\gamma_L$ between $N(L)=N(L_j^n)$ and the vector $(1,0)$ is positive, i.e., $0<\gamma_L\leq\pi$.
Then, if for some given $i=(i_1,\ldots,i_{M_n^2})\in I_n$, we have $\Sigma_{M_n^2}(i_1,\ldots,i_{M_n^2})\leq -M_n\alpha$, this implies that some segment $S_{i_1,\ldots,i_k}$ has rotated clockwise by an angle equal to $M_N\alpha = \gamma_L$ with respect to $L$, so that it is vertical and $N(S_{i_1,\ldots,i_k})=(1,0)$. By construction this also ensures that $N(S_{i_1,\ldots,i_{M_n^2}})=(1,0)$. Consequently,
$$\#\cV(\Gamma_j^n)\geq \#\big\{i\in I_n:\Sigma_{M_n^2}(i)\leq -M_n\alpha\big\}$$
and thus
\begin{equation}\label{eqVV1}
\frac{\#\cV(\Gamma_j^n)}{\# I_n} \geq \mu_{I_n}\big(\{i\in I_n:\Sigma_{M_n^2}(i)\leq -M_n\alpha\}\big).
\end{equation}

Observe that
$$\frac{\Sigma_{M_n^2}}{M_n\,\sigma} = \frac{\sqrt2\,\Sigma_{M_n^2}}{M_n\,\alpha}.$$
By the central limit theorem, the random variable 
$\frac{\Sigma_{M_n^2}}{M_n\,\sigma}$ converges in law to the standard normal distribution as $M_n\to\infty$.
Therefore, for some absolute constant $\eta>0$ (determined by the standard normal distribution),
$$
 \mu_{I_n}\big(\{i\in I_n:\Sigma_{M_n^2}(i)\leq -M_n\alpha\}\big) = \mu_{I_n}\Big(\Big\{i\in I_n:\frac{\Sigma_{M_n^2}(i)}{M_n\,\sigma}\leq -\sqrt2\Big\}\Big) \to \eta \quad\mbox{ as $M_n\to\infty$.}
 $$
Consequently, by \rf{eqVV1}, 
$$\frac{\#\cV(\Gamma_j^n)}{\# I_n} \geq \frac\eta2,$$
for $M_n$ big enough.

To finish the proof of \rf{eqsumguai}, recall that any segment $S_i$, with $i\in I_n$, satisfies
$\HH^1(S_i)\geq 4^{-M_n^2}\,\HH^1(L_j^n)$. Thus,
$$\sum_{S\in \mathcal \cV(\Gamma_j^n)}\HH^1(S) \geq \#\cV(\Gamma_j^n)\,4^{-M_n^2}\,\HH^1(L_j^n)  \geq \frac\eta2\,\# I_n\,4^{-M_n^2}\HH^1(L_j^n) =\frac\eta2\,\HH^1(L_j^n)\geq c\,\eta\,\HH^1(\Gamma_j^n),$$
using \rf{eqcorarc10} for the last inequality.
\end{proof}

\vv
\subsection{The domain $\Omega_{n+1}$}

Recall that $\partial\Omega_n= \bigcup_{j=1}^{m_n} L_j^n$ and that each segment $L_j^n$ has the same end-points as
the chord-arc curve $\Gamma_j^n$. We consider the curve
$$G_{n+1} = \bigcup_{j=1}^{m_n} \Gamma_j^{n}.$$
The fact that the curves $\Gamma_j^n$ are chord-arc and the property \rf{eqgammah23} will ensure that $G_{n+1}$ is a quasicircle.

Let $V_{n+1}$ be the domain enclosed by $G_{n+1}$. 
This domain should be considered as a first approximation of $\Omega_{n+1}$. In order to guaranty the vanishing Reifenberg flatness property of the final domain $\Omega$, 
we have to modify $V_{n+1}$. Indeed, it is easy to check that the angle between the outer normals for $V_{n+1}$ at any two neighboring segments contained in the same piecewise curve $\Gamma_j^{n}$ is at most $\alpha_{L_j^n}\leq C/M_n$, and we will choose $M_n$ so that $M_n\to\infty$.
However, the angle between two consecutive segments of the curve $G_{n+1}$ which are contained in different
(but neighboring) curves $\Gamma_j^n$, $\Gamma_{j+1}^n$ is the same as the angle between $L_j^n$ and $L_{j+1}^n$, which is not convenient, because we want this to become flatter. 

To define a suitable smoothened version of $G_{n+1}$ we need some additional notation.
Let $T_1,\ldots,T_{t_{n+1}}$ the segments which compose the piecewise curve $G_{n+1}$. That is to say,
$$\big\{T_i\big\}_{1\leq i\leq t_{n+1}} = \big\{S_{i_1,\ldots,i_{M_n^2}}(\Gamma_j^n):1\leq j\leq m_n,\,1\leq i_1,\ldots,i_{M_n^2}\leq 4 \},$$
where $S_{i_1,\ldots,i_{M_n^2}}(\Gamma_j^n)$ stands for the segment $S_{i_1,\ldots,i_{M_n^2}}$ appearing in the construction of $\Gamma_j^n$.
Suppose that  $T_1,\ldots,T_{t_{n+1}}$
are ordered clockwise in the curve $G_{n+1}$. Let 
\begin{equation}\label{eqan+1}
a_{n+1}=\min_{1\leq i\leq t_{n+1}}\HH^1(T_i),
\end{equation}
and for each $i$ let $\wt T_i$ be a closed segment such that $\wt T_i\subset T_i$ with the same mid-point as $T_i$ and such that 
\begin{equation}\label{eqan+21}
\HH^1(\wt T_i)= \HH^1(T_i) - \frac {a_{n+1}}{200}.
\end{equation}
%That is, $\wt T_i =\frac{99}{100}T_i$.
Let $C_i$ be a closed circular arc that joints $\wt T_i$ to $\wt T_{i+1}$, so that $C_i$ is tangent both to $T_i$ and $T_{i+1}$ and the tangency points coincide with the two closest end-points of $\wt T_i$ and $\wt T_{i+1}$, as in Figure \ref{figure3}. Of course, the
arc $C_i$ is determined once the end-points of each segment $\wt T_i$ are fixed by the condition \rf{eqan+21}. Also, by elementary geometry, it is clear that the angle subtended by $C_i$ is equal to the supplementary of the angle formed by the two consecutive segments
$T_i$ and $T_{i+1}$.
Then we let
$$\wt G_{n+1} = \bigcup_{1\leq i \leq t_{n+1}} (\wt T_i \cup C_i).$$
That is, roughly speaking, $\wt G_{n+1}$ is the curve obtained by erasing $T_i\setminus \wt T_i$, for $1\leq i\leq t_{n+1}$ and replacing the erased parts by circular arcs, so that the resulting curve is of type $C^1$.
We remark that the precise construction of $\wt G_{n+1}$ from $G_{n+1}$ is not very relevant. What it is important is that $\wt G_{n+1}$
is a $C^1$ quasicircle (with a uniform character), close in Hausdorff distance to $G_{n+1}$, and it coincides with $G_{n+1}$ in a big portion of each segment $T_i$.

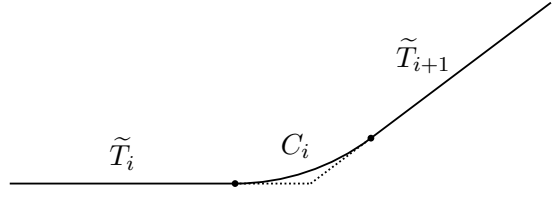
\begin{figure}
\begin{tikzpicture}[radius=30mm,line width=0.7pt]
%\draw (0,0) -- (40mm,0) -- (71.945mm,24.073mm);
\draw (0,0) -- (29.961mm,0);
\draw[densely dotted] (29.961mm,0) -- (40mm,0);
\draw[densely dotted] (40mm,0) -- (48.017mm,6.041mm);
\draw (48.017mm,6.041mm) -- (71.945mm,24.073mm);
%\draw (40mm,0) -- (71.945mm,24.073mm);
  \draw (29.961mm,0)  arc[start angle=270, end angle=307]; 
  \fill (29.961mm,0) circle (0.45mm);
  \fill (48.017mm,6.041mm) circle (0.45mm);
  \draw (15mm,4mm) node {$\wt T_i$};
  \draw (55mm,17mm) node {$\wt T_{i+1}$};
  \draw (38mm,5mm) node {$C_i$};
\end{tikzpicture}
\vv

\caption{The segments $\wt T_i$, $\wt T_{i+1}$, and the arc $C_i$ that joins them. The length of each dotted segment
is $a_{n+1}/100$.}\label{figure3}
\end{figure}

The next step is to consider a family of points $z_1^{n+1},\ldots,z_{m_{n+1}}^{n+1}$ from $\wt G_{n+1}$ ordered clockwise,  with $m_{n+1}\geq t_{n+1}$, so that
$$|z_i^{n+1}-z_{i+1}^{n+1}| = |z_j^{n+1}-z_{j+1}^{n+1}|=:\ell_{n+1}\quad \mbox{for $1\leq i,j\leq m_{n+1}$,}$$
 understanding $z_{m_{n+1}+1}^{n+1} = z_1^{n+1}$. We denote $L_i^{n+1} = [z_i^{n+1},z_{i+1}^{n+1}]$.
 Further, we take $m_{n+1}$ large enough so that
 $\ell_{n+1}\leq \frac {a_{n+1}}{200}$ and so that the supplementary of the angle between two neighboring segments $L_i^{n+1}$, $L_{i+1}^{n+1}$ is at most $2^{-n-4}\pi$ (here we mean the angle equal to the smallest one of the two supplementary angles that form the two lines supporting $L_i^{n+1}$ and $L_{i+1}^{n+1}$).
 By a continuity argument and the $C^1$ character of $\wt G_{n+1}$ it is easy to prove the existence of the family of points  $z_1^{n+1},\ldots,z_{m_{n+1}}^{n+1}$. 
 
 Finally, we let $\Omega_{n+1}$ be the domain enclosed by the curve formed by the union of the segments 
 $L_1^{n+1},\ldots,L_{m_{n+1}}^{n+1}$, so that
 $$\partial\Omega_{n+1} = \bigcup_{1\leq i \leq m_{n+1}} L_i^{n+1}.$$

\vv

\subsection{The domain $\Omega$}

It is easy to check that the curves $\partial\Omega_n$ are quasi-circles which converge in Hausdorff distance to another quasi-circle $\Gamma_\infty$ as $n\to\infty$. We let $\Omega$ be the domain bounded by the quasi-circle $\Gamma_\infty$, and 
we also set $\Omega^+=\Omega$, $\Omega^-=\R^2\setminus \overline{\Omega^+}$.

Recall that, for any $n\geq1$, the curve $\partial\Omega_n$ is piecewise linear and that the supplementary angle between two adjacent segments $L_i^{n+1}$, $L_{i+1}^{n+1}$ which compose $\partial\Omega_n$ is at most $2^{-n-4}\pi$. From this fact, and the construction above, using also Lemma \ref{lemflat}, one can check that $\Omega$ is a vanishing Reifenberg-flat domain. Further, from \rf{eqgjnalf} and the construction above, it follows easily that
\begin{equation}\label{eqinc99}
\partial\Omega\subset\UU_{c\ell_n/M_n}(\partial\Omega_n),
\end{equation}
for some fixed $c>0$.

\vvv

% ******************************************************************************************************

Recall also that $T_1,\ldots,T_{t_{n+1}}$ are the segments which compose the piecewise curve $G_{n+1}$.
%By the construction of $\wt G_{n+1}$ and $\Omega_{n+1}$, it follows easily that $\frac9{10} T_i\subset \partial \Omega_{n+1}$ for each $1\leq i\leq t_{n+1}$.
Denote by $\cV_{n+1}$ the subfamily of segments $T\in\{T_1,\ldots,T_{t_{n+1}}\}$ such that $N(T)=(1,0)$,
and let
\begin{equation}\label{eqF}
 F =  \bigcup_{n\geq1} \,\bigcup_{T\in \cV_{n+1}}  	\tfrac12T.
 \end{equation}

\begin{lemma}\label{lemFdins}
The set $F$ is contained in $\partial\Omega$. 
\end{lemma}

Remark that we cannot ensure that the larger set
$$ F' =  \bigcup_{n\geq1} \,\bigcup_{T\in \cV_{n+1}}  	T$$
is contained in $\partial\Omega$ because in the smoothening process, i.e., the construction of $\wt G_{n+1}$ from $G_{n+1}$, some portion of the vertical segments $T_i$ is deleted in order to obtain $\wt T_i$. Also, for $m\geq n+1$, in the construction of $\wt G_{m+1}$ from $G_{m+1}$, more portions of the segments $T_i$
may be deleted.
 However, 
the deleted portion of the segments $T_i$ is pretty small for each $m\geq n$ (because of the choice of the parameter $a_{m+1}/200$ in \rf{eqan+21}), so that we can ensure that $F\subset \partial\Omega$. This what the lemma claims.

\begin{proof}
Let $T\in \cV_{n+1}$. Then $T\subset G_{n+1}$ and $\frac{99}{100}T\subset\wt G_{n+1}$.
Since we chose $\ell_{n+1}\leq \frac {a_{n+1}}{200}$, it follows easily that $\frac{9}{10}T\subset\partial\Omega_{n+1}$, and moreover $\frac{9}{10}T$ is covered by a subfamily
of segments from $L_i^{n+1}$ such that $N(L_i^{n+1})=(1,0)$. We denote this family by
$\{L_i^{n+1}\}_{i\in I^{n+1}(T)}$.
By inspection of the above construction, it can be seen that (assuming the sequence $\{m_n\}_n$ to growth fast enough if necessary) at most two segments from the family $\{L_i^{n+1}\}_{i\in I^{n+1}(T)}$ are not included in $\partial\Omega_{n+2}$.
Iterating, we see for each $m>n$, there is a segment $T^m$ concentric with $T$ contained in $\partial\Omega_m$ with length
$$\HH^1(T^m)\geq \frac9{10}\,\HH^1(T)- 2\sum_{j={n+1}}^m \ell_j\geq \frac12\,\HH^1(T).$$
Letting $m\to\infty$, this implies that $\frac12T\subset\partial\Omega$.
\end{proof}
\vv

\begin{lemma}\label{lembigpieces}
For any ball $B$ centered in $\partial\Omega$ such that $\diam(B)\leq\diam(\partial\Omega)$ there are two chord-arc
subdomains $\Omega_B^\pm\subset \Omega^\pm$ such that
\begin{equation}\label{eqH1F}
\HH^1(F\cap \partial\Omega_B^+\cap \partial\Omega_B^-) \geq c\,\diam(B),
\end{equation}
where $c>0$ is an absolute constant and the chord-arc character of $\Omega_B^\pm$ does not depend on $B$.
\end{lemma}

To prove this lemma we will use following result of Jonas Azzam \cite[Theorem 6.4]{Azzam}:

\begin{theorem}\label{teo-azz}
Let $\Omega^+\subset\R^{n+1}$ be a two-sided NTA domain and let $\Omega^-=(\overline \Omega)^c$. Let $E\subset \partial\Omega \cap Z$, where 
  $Z$ is a uniformly $n$-rectifiable set. Then there are two chord-arc domain $\Omega^\pm_E$ such that $\Omega_E^\pm
  \subset \Omega^\pm$, with $\diam(\partial\Omega_E^\pm)\gtrsim\diam(\partial\Omega^+)$, such that $E\subset\partial\Omega^+\cap\partial\Omega_E^+\cap\partial\Omega_E^-$.
  The chord-arc parameters of $\Omega^\pm_E$ depend only on the NTA parameters of $\Omega^\pm$ and on the uniform
  rectifiability parameters of $Z$.
\end{theorem}
\vv

\begin{proof}[Proof of Lemma \ref{lembigpieces}]
Given a ball $B$ as in the lemma, let $n$ be such that $10\ell_n\geq r(B)>10\ell_{n+1}$. 
As shown in Lemma \ref{lembilip}, each curve $\Gamma_j^n$ is $C$-Ahlfors regular, for some fixed absolute constant $C$. Since, for $n$ fixed,  only a finite number of the curves $\Gamma_j^n$ intersect $2B$, it follows that $G_n\cap 2B$ is also Ahlfors regular (perhaps with a slightly worse constant $C'$). By construction, this implies that also $\wt G_n\cap 1.5B$ and $\partial \Omega_{n+1}\cap B$ are upper $C''$-Ahlfors regular.

Let $H$ be the arc of $B\cap\partial\Omega_{n+1}$ that passes through the center of $B$. Then $H$ is a $C''$-Ahlfors regular curve with $\diam(H)\approx \HH^1(H)\geq 2\,r(B)$. Denote by $\{L_i^{n+1}\}_{i\in J_H}$ the family of segments $L_i^{n+1}$, $1\leq i\leq m_{n+1}$, such that
$L_i^{n+1}\subset H$. By the choice of $n$, it follows easily that
\begin{equation}\label{eqsumgi*}
\sum_{i\in J_H} \HH^1(L_i^{n+1}) \approx r(B).
\end{equation}
Next, form the curve $\Gamma_H$ by replacing each $L_i^{n+1}$, $i\in J_H$, in the arc $H$ by the corresponding curve 
$\Gamma_i^{n+1}$. Since $\HH^1(\Gamma_i^{n+1})\approx \HH^1(L_i^{n+1})$ and $\Gamma_i^{n+1}$ is Ahlfors regular for each $i$, we infer that $\Gamma_H$ is a $C'''$-Ahlfors regular curve, where $C'''$ is another absolute constant.

By Lemmas \ref{lemvertical} and \rf{lemFdins}, we deduce that
$$\HH^1(\Gamma_i^{n+1}\cap F \cap \partial \Omega) \gtrsim \HH^1(L_i^{n+1})\quad \mbox{ for all $i\in J_H$.}$$
Summing on $i\in J_H$ and using \rf{eqsumgi*}, we deduce that
$$\HH^1(\Gamma_H\cap F \cap \partial \Omega) \gtrsim r(B).$$
Finally we apply Theorem \ref{teo-azz} with $E=F\cap B$ and $Z=\Gamma_H$ to deduce the existence of the required chord-arc domains
$\Omega_B^\pm\subset\Omega^\pm$ satisfying \rf{eqH1F}.
\end{proof}
\vvv

% ******************************************************************************************************

\section{The set $F$ has full harmonic measure and \rf{eqnb*} does not hold}

\begin{lemma}\label{lemfull}
The set $F$ defined in \rf{eqF} has full harmonic measure in $\Omega^+$.
\end{lemma}

Notice that the outer unit normal of $\Omega^+$ at every $\xi\in F$ equals $(1,0)$. From the fact that $F$ has full harmonic measure, it follows then that the outer unit normal of $\Omega^+$ is constant $\omega^+$-a.e.
in $\partial\Omega^+$.

\begin{proof}[Proof of Lemma \ref{lemfull}]
By standard arguments, it suffices to prove that there is a constant $c>0$ such that
for any ball $B$ centered in $\partial\Omega^+$ such that $\diam(B)\leq\diam(\partial\Omega^+)$, it holds
$$\omega^+(F\cap B) \geq c\,\omega^+(B).$$
Here we assume that the pole for harmonic measure is a point $p\in\Omega$ such that $\dist(p,\partial\Omega)\gtrsim
\diam(\Omega^+)$.

Let $\Omega_B^\pm$ be as in Lemma \ref{lembigpieces} and let $p_B\in B\cap \Omega_B^+$ be such that
$$\dist(p_B, \partial\Omega_B^+)\gtrsim \diam(B),$$
so that $p_B$ is also a corkscrew point for $B$ with respect to $\Omega^+$.
By the change of pole formula in Theorem \ref{teounif}, it suffices to show that
$$\omega^{+,p_B}(F\cap B) \gtrsim 1.$$
By \rf{eqH1F}, we know that
$$\HH^1(F\cap \partial\Omega_B^+) \geq c\,\diam(B) \approx \HH^1(\partial\Omega_B^+ \cap B).$$
Since the harmonic measure $\omega_{\Omega_B^+}^{p_B}$ is an $A_\infty$ weight with respect to $\HH^1|_{\partial
\Omega_B^+}$, we deduce that
$$\omega_{\Omega_B^+}^{p_B}(F\cap B) \gtrsim 1.$$
Then, by the maximum principle it holds
$$\omega^{+,p_B}(F\cap B) \geq\omega_{\Omega_B^+}^{p_B}(F\cap B) \gtrsim 1,$$
as wished.
\end{proof}

\vv
To complete the proof of Theorem \ref{teoex} it only remains to show that the condition \rf{eqnb*} does not hold.
To this end, it suffices to check that there are arbitrarily small balls $B$ centered in $\partial\Omega^+$
such that 
\begin{equation}\label{eqNB01}
|N_B-(1,0)|\geq c_2,
\end{equation}
for some fixed $c_2>0$.

For simplicity, suppose that one of the initial segments $L_j^0$ from $\Omega_0$ is horizontal and satisfies
$N(L_j^0) =(0,1)$. By construction, the segment with coding $S_{1,\ldots,1}$ (with $M_0^2$ 1's) is also horizontal and its associated outer normal is $(0,1)$. So one of the segments $T_i$, $1\leq i\leq t_1$, from the curve $G_1$ is also horizontal and $N(T_i)=(0,1)$. By the smallness of the parameter $a_1$ defined in \rf{eqan+1} and the fact that $\ell_1\leq a_1/100$, it follows easily that there exists at least one segment $L_k^1$ from $\partial\Omega_1$ contained in $T_i$, so that $N(L_k^1)=(0,1)$. Iterating, we deduce that for any $n\geq1$, there is a horizontal segment $L_{k_n}^n$ from $\partial\Omega_n$ such that $N(L_{k_n}^n)=(0,1)$.

Let $B_n$ be a ball centered in the mid-point of $L_{k_n}^n$ with radius $r_n:=\HH^1(L_{k_n}^n)/4$.
From \rf{eqinc99}, we deduce that 
$\partial\Omega\cap \frac12 B_n\neq \varnothing$ and $\partial\Omega\cap B_n\subset \UU_{r_n/100}(L_{k_n}^n)$ for $n$ big enough.
Now let $\wt B_n\subset B_n$ be a ball centered in $\partial\Omega$ such that $\diam(\wt B_n) \geq \frac9{10}\diam(B_n)$.
Then the vector $N_{\wt B_n}$ is close to being vertical and thus $|N(\wt B_n)- (1,0)|\gtrsim1$.
Hence there are arbitrary small balls satisfying \rf{eqNB01}.

% ******************************************************************************************************

% ******************************************************************************************************

\vv

\end{document}